\documentclass[11pt]{article}

\usepackage[
	a4paper,
	margin=1in
]{geometry} 
\usepackage{setspace}
\usepackage{sectsty} 
\usepackage{titlesec}
\usepackage{appendix}

\usepackage{amsmath, amssymb, amsthm}
\usepackage{mathtools}

\usepackage[
	setpagesize=false,
	colorlinks=true,
	linkcolor=BrickRed,
        citecolor=OliveGreen,
        urlcolor=black,
	pdfencoding=auto,
	psdextra,
]{hyperref}
\usepackage{cleveref}

\usepackage{etoolbox} 
\usepackage[shortlabels]{enumitem}
\usepackage{xparse} 

\usepackage{graphicx}
\usepackage[dvipsnames]{xcolor}
\usepackage{tikz}

\usepackage{todonotes}
\usepackage[
    deletedmarkup=sout,
    commentmarkup=uwave,
    authormarkuptext=name
]{changes}

\usepackage{comment}

\setlist[enumerate,1]{label=(\arabic*), ref=(\arabic*)}
\setlist[enumerate,3]{label=(\roman*), ref=(\roman*)}

\allsectionsfont{\boldmath} 

\titlelabel{\thetitle.\quad}

\theoremstyle{plain}
\newtheorem{theorem}{Theorem}[section]
\newtheorem{lemma}[theorem]{Lemma}

\newtheorem{proposition}[theorem]{Proposition}

\newtheorem{conjecture}[theorem]{Conjecture}

\newtheorem{case}{Case}
\makeatletter
\@addtoreset{case}{section}
\makeatother

\newtheorem{claim}[theorem]{Claim}
\newtheorem*{claim*}{Claim}

\makeatletter
\newenvironment{claimproof}[1][Proof]{\par
	\pushQED{\qed}%
	
	\normalfont \topsep6\p@\@plus6\p@\relax
	\trivlist
	\item[\hskip\labelsep
	\textit{#1}\@addpunct{.}~]\ignorespaces
}{%
	\popQED\endtrivlist\@endpefalse
}
\makeatother

\theoremstyle{definition}
\newtheorem{definition}[theorem]{Definition}
\newtheorem*{definition*}{Definition}

\newcommand{\ZZ}{\mathbb{Z}}
\newcommand{\EE}{\mathbb{E}}
\newcommand{\PP}{\mathbb{P}}

\newcommand{\calF}{\mathcal{F}}
\newcommand{\calG}{\mathcal{G}}
\newcommand{\calH}{\mathcal{H}}

\newcommand{\calS}{\mathcal{S}}

\newcommand{\calU}{\mathcal{U}}
\newcommand{\calV}{\mathcal{V}}

\newcommand{\abs}[1]{\left\lvert#1\right\rvert}
\newcommand{\Cay}{\mathrm{Cay}}

\newcommand{\defeq}{\coloneqq}

\usetikzlibrary{matrix,arrows,decorations.pathmorphing}
\usetikzlibrary{positioning}
\usetikzlibrary{graphs} 
\usetikzlibrary{graphs.standard}

\newcounter{propcounter}
\stepcounter{propcounter}

\title{Random Cayley sum hypergraphs and $k$-fold sumsets}

\author{
Jihyo Chae%
        \thanks{Department of Mathematics, Yonsei University, South Korea. 
        E-mail: {\ttfamily jihyochae@yonsei.ac.kr.} Supported by the Samsung STF Grant SSTF-BA2201-02, and the National Research Foundation of Korea (NRF) grant funded by the Korea Government(MSIT) NRF-2022R1C1C1010300} \and 
Hyunwoo Lee%
        \thanks{Department of Mathematical Sciences, KAIST, South Korea and Extremal Combinatorics and Probability Group (ECOPRO), Institute for Basic Science (IBS).
        E-mail: {\ttfamily hyunwoo.lee@kaist.ac.kr.} Supported by the National Research Foundation of Korea (NRF) grant funded by the Korea government(MSIT) No. RS-2023-00210430, and the Institute for Basic Science (IBS-R029-C4).}
}

\usepackage[square,sort,comma,numbers]{natbib}
\begin{document}
\maketitle


\begin{abstract}
    We denote by $f_k(\Gamma)$ the largest integer with the property that every subset of a finite abelian group $\Gamma$ of size at least $|\Gamma| - f_k(\Gamma)$ is a $k$-fold sumset. Extending a recent result of Alon and Pham, we prove that
    $$
        f_k(\Gamma) \leq \widetilde{O} \left(n^{(2k-1)/(4k-3)}\right)
    $$
    holds for all finite abelian groups $\Gamma$ and integers $k \geq 2$, where $n = |\Gamma|$. Additionally, we also show that the lower bound  $f_k(\Gamma) \geq \widetilde{\Omega} \left(n^{1/k}\right)$ holds if $\Gamma$ has no nontrivial element of order dividing $k$. Our upper bound improves a previous result of Balogh, Liu, and Sharifzadeh, and recovers the bound of Alon and Pham in the case $k = 2$. The proof relies on a new upper bound for the independence number of random Cayley sum hypergraphs, which may be of independent interest.
\end{abstract}


\section{Introduction}\label{sec:intro}

The study of additive structures in a subset of a given abelian group has been of great interest. In particular, analyzing the behavior of sets with a given additive structure from multiple perspectives has led to the development of various problems in additive combinatorics, most of which concern sumsets in an additive group.

For a given abelian group $\Gamma$ and its subset $A\subseteq \Gamma$, we define the \emph{sumset} of $A$ as the set $\{ a_1 + a_2: a_1, a_2\in A \} \subseteq \Gamma$ and denote as $A + A$ or $2A$. Similarly, for an integer $k \geq 2$, we define \emph{$k$-fold sumset} of $A$ as the set $\{\sum_{i\in [k]} a_i: a_i\in A \}$ and denote as $k A$.
An important quantity related to the additive structure of a set is the \emph{doubling constant}, defined as the ratio $\frac{|2A|}{|A|}$ for a given finite set $A\subseteq \Gamma$. The fundamental theorem that connects the doubling constant and the additive structure of the set is Freiman's theorem \cite{FR64}, which says that if a subset $A$ has a small doubling constant, then $A$ contains a rich additive structure, such as arithmetic progressions.

Since ($k$-fold) sumsets play a central role in additive combinatorics, it is natural to ask the size and number of sumsets in a given group. One observation of Green \cite{Green-essay} was that it is easier to find a representation of the form $S=A+A$ whenever $S$ is large. It was therefore natural to ask whether all large sets in a finite group are sumsets. Green answered this question positively and provided a bound on the threshold for this phenomenon. We consider a higher uniformity version of Green's question. For an integer $k \geq 2$, denote by $f_k(\Gamma)$ the largest integer with the property that every subset of $\Gamma$ of size at least $|\Gamma| - f_k(\Gamma)$ is a $k$-fold sumset.

More than twenty years ago, Green~\cite{Croot-Lev,Green-essay} showed that the bound $f_2(\ZZ_p) \geq \Omega(\log p)$ holds for all sufficiently large prime $p$, and consequently raised a problem to determine the value $f_2(\ZZ_p)$.
Later, Alon~\cite{A07} improved the lower bound to $f_2(\ZZ_p)\geq \tilde{\Omega}(p^{1/2})$ and Green and Gowers~\cite{B17}, and Alon~\cite{A07} provided the upper bound $f_2(\ZZ_p)\leq \tilde{O}(p^{2/3})$. Recently, the upper bound was considerably improved and extended to every finite abelian group as follows.

\begin{theorem}[Alon--Pham~\cite{AH25-arxiv}]\label{thm:prev}
    For all finite abelian group $\Gamma$ of order $n$, we have $f_2(\Gamma) \leq \widetilde{O}(n^{3/5})$.
\end{theorem}

For $k$-fold sumsets, the previous best known upper bound of $f_k(\ZZ_p)$ is obtained by Balogh, Liu, and Sharifzadeh in the appendix from the arXiv version of~\cite{BLS}. 

\begin{theorem}[Balogh--Liu--Sharifzadeh~\cite{Balogh-Liu-Sharifzadeh-arxiv}]\label{thm:bls}
    For all prime $p$ and integer $k \geq 2$, we have
    $$
        f_k(\ZZ_p)\leq O\left(p^{k/(2k-1)}\right).
    $$
\end{theorem}

In this paper, we obtain a polynomial improvement of Theorem~\ref{thm:bls} by extending the upper bound of Theorem~\ref{thm:prev} to all $k$, and also provide a lower bound as follows. 

\begin{theorem}\label{thm:main}
    For all integer $k \geq 2$ and a finite abelian group $\Gamma$ of $n$, we have
    $$
        f_k(\Gamma)\leq \widetilde{O}\left(n^{(2k-1)/(4k-3)}\right).
    $$ 
    If $\Gamma$ has no nontrivial element of order dividing $k$, then we also have
    $$
        f_k(\Gamma) \geq \widetilde{\Omega}\left(n^{1/k}\right).
    $$
\end{theorem}

The main idea in the proof of Theorem~\ref{thm:prev} was to use an improved upper bound on the independence number of a random Cayley sum graph. To extend this approach to the hypergraph setting, which is more suitable for studying $k$-fold sumsets, we introduce the \emph{Cayley sum $k$-graph} of a finite abelian group $\Gamma$ with generating set $S \subseteq \Gamma$, denoted by $\Cay_k^+(\Gamma; S)$. Its vertex set is $\Gamma$, and its edge set consists of all $k$-multisets $\{a_1,\ldots,a_k\}$ such that $a_1+\cdots+a_k \in S$.
In contrast to general $k$-uniform hypergraphs, this Cayley sum $k$-graph permits repeated vertices within an edge; that is, each edge is regarded as a multiset rather than a set. For a real number $0 \leq p \leq 1$, we denote by $\Cay_k^+(\Gamma; p)$ the random Cayley sum $k$-graph with the generating set $S$, which is a $p$-random subset of $\Gamma$, that is, each element of $\Gamma$ is chosen as a member of $S$ with probability $p$ independently.

In the last decade, there have been extensive studies on the independence number of random Cayley (sum) graphs, for instance~\cite{A07,A13,CFPY,CDM,AH25-arxiv}. To the best of our knowledge, the current best upper bound on the independence number of a random Cayley sum graph for small values of $p = o(1)$ is due to Alon and Pham~\cite{AH25-arxiv}.

\begin{theorem}[Alon--Pham~\cite{AH25-arxiv}]\label{thm:Alon-Pham-indep}
    Let $\Gamma$ be a finite abelian group of order $n$ and $p \in (0, 1/2]$ be a real number. Then we have
    $$
        \alpha(\Cay_2^+(\Gamma; p)) \leq \widetilde{O}\left( p^{-3/2} \right)
    $$
    with high probability.
\end{theorem}

To prove Theorem~\ref{thm:main}, we generalize Theorem~\ref{thm:Alon-Pham-indep} to the hypergraph setting. As discussed earlier, we also handle asymmetric cases, which leads us to consider the \emph{$k$-partite hole number} instead of the independence number. 
The $k$-partite hole number of $\Cay_k^+(\Gamma; S)$, denoted by $\alpha_k(\Cay_k^+(\Gamma; S))$, is defined as the largest integer $m$ for which there exist subsets $U_1, \dots, U_k \subseteq \Gamma$, each of size $m$, such that there is no edge $\{a_1, \dots, a_k\}$ in $\Cay_k^+(\Gamma; S)$ with $a_i \in U_i$ for every $i \in [k]$. From its definition, it is clear that $\alpha(\Cay_k^+(\Gamma; S)) \leq \alpha_k(\Cay_k^+(\Gamma; S))$. We then show that random Cayley sum $k$-graphs typically have small $k$-partite hole number, as stated below. This result generalizes and slightly strengthens Theorem~\ref{thm:Alon-Pham-indep}.

\begin{theorem}\label{thm:ind}
    Let $\Gamma$ be a finite abelian group of order $n$ and $p \in (0, 1/2]$ be a real number. Then for each $k \geq 2$, we have
    $$
        \alpha_k(\Cay_k^+(\Gamma; p)) \leq \widetilde{O}\left( p^{-(2k-1)/(2k-2)} \right)
    $$
    with high probability.
\end{theorem}

\paragraph{Remark.} In \cite{BLM}, independently, Balogh,  Lee, and Methuku proved similar results.

\paragraph{Organization.} In Section~\ref{sec:pre}, we briefly demonstrate the approach of \cite{AH25-arxiv} to the proof of \Cref{thm:prev,thm:Alon-Pham-indep} and introduce some basic tools that will be used throughout the paper. In Section~\ref{sec:lb}, we prove the lower bound of Theorem~\ref{thm:main} via the probabilistic method. Then, we develop the higher uniformity version of the covering lemma, which leads to the proof of the upper bound of Theorem~\ref{thm:main}.


\section{Preliminaries}\label{sec:pre}
Throughout the remaining paper, we assume $\Gamma$ to be a finite abelian group of order $n$, and all logarithms are taken to base~$2$.


\subsection{Probabilistic inequalities}\label{subsec:prob-ineq}

In the proof of our main theorems, we use several concentration inequalities. The following is Hoeffding's inequality.

\begin{lemma}[Hoeffding's inequality~\cite{H63}]\label{lem:hoeffding}
    Let $X_1, \dots, X_m$ be a collection of independent random variables such that $a_i \leq X_i \leq b_i$ for each $i \in [m]$ and define $X \defeq \sum_{i \in [m]}X_i$. Then for any positive real number $t$, the following holds.
    \[
        \PP[X-\EE[X]\geq t] \leq \exp \left( -\frac{2t^2}{\sum_{i \in [m]}(b_i-a_i)^2} \right).
    \]
\end{lemma}

In Section~\ref{sec:proof}, we use the following version of Janson's inequality and Paley--Zygmund inequality stated in \cite{AS16} to find a probabilistic cover in the sparse regime. 

\begin{lemma}[Janson's Inequality]\label{lem:jan}
    Let $\Omega$ be a finite set and $A_i$ $(i \in I)$ be subsets of $\Omega$. Define $R$ as a random subset of $\Omega$ constructed by choosing each element $x$ of $\Omega$ with probability $p_x$. Define $B_i$ as the event $ A_i \subseteq R $ for each $i \in I$ and $X_i$ as the indicator random variable for $B_i$. Also, say that $i \sim j$ if and only if both $i \neq j$ and $A_i \cap A_j \neq \emptyset$ holds. Then for $X \defeq \sum X_i$, $\mu \defeq \EE [X]$, and $\Delta \defeq \sum_{i \sim j}\PP [B_i\and B_j]$, the following holds.
    \begin{enumerate}
        \item If $\PP [B_i] \le \epsilon$ for each $i \in I$, 
        \[
        \PP\left[\bigwedge_{i \in I}B_i^c\right] \le \exp{(-\mu + \Delta/2)}.
        \]
        \item If $\Delta \ge \mu$,
        \[
        \PP \left[\bigwedge_{i \in I}B_i^c \right] \le \exp{(-\mu^2 / 2\Delta)}.
        \]
        
    \end{enumerate}
    
\end{lemma}

\begin{lemma}[Paley--Zygmund Inequality]\label{lem:pzy}
    Let $Z$ be a nonnegative random variable with finite variance. Then for any $\theta \in [0,1 ]$, 
    \[
    \PP [Z>\theta \EE (Z)] \ge (1-\theta)^2 \frac{\EE[Z]^2}{\EE [Z^2]}.
    \]
    In particular, we have
    $$
        \PP[Z > 0] \geq \frac{\EE[Z]^2}{\EE [Z^2]}.
    $$
    
\end{lemma}

In Section~\ref{sec:lb}, we use another version of Janson's inequality stated in \cite{Y22} that bounds lower-tail probability to bound the edge count of a random induced subhypergraph from below. For additional information on Janson's inequality, see also \cite{AS16}.

\begin{lemma}\label{lem:jan2}
    Under the same assumptions as in Lemma~\ref{lem:jan}, 
    \[
    \PP [X \le \mu - t] \le \exp \left(- \frac{t^2}{2(\mu + \Delta)} \right)
    \]
    holds for any $t \in [0, \mu]$.
\end{lemma}


\subsection{Discrete Fourier Analysis}\label{subsec:Disc-FA}

Both of the proofs of Theorem~\ref{thm:prev} and our main theorems use the Fourier analytic method. Hence, we summarize the properties of the discrete Fourier analysis that we need. 

Let $\Gamma$ be a finite abelian group, and $\hat{\Gamma}$ denote its dual group. Given a set $S \subseteq \Gamma$, denote its characteristic function as $S$.

\begin{definition}
    For a function $f: \Gamma \rightarrow \mathbb{C},$ its \emph{Fourier transform} $\hat{f} : \hat{\Gamma} \rightarrow \mathbb{C}$ is defined as 
    $$
        \hat{f}(\chi) \defeq \EE_{x \in \Gamma}f(x)\chi(x).
    $$
\end{definition}

\begin{definition}
    Let $f, g: \Gamma \rightarrow \mathbb{C}.$ Define the \emph{convolution} as $(f*g)(x) := \EE_y[f(y)g(x-y)].$
\end{definition}

We state some basic identities of Fourier analysis in finite abelian groups below.

\begin{proposition}\label{prop:fourier}
    Let $f, g: \Gamma \rightarrow \mathbb{C}$. Then the following identities hold for all $x\in \Gamma$ and $\chi \in \hat{\Gamma}$.
    \begin{enumerate}
        \item $\widehat{f*g}(\chi) = \hat{f}(\chi) \hat{g}(\chi)$ [Convolution identity] 
        \item $f(x)=\sum_{\chi}\hat{f}(\chi)\chi(-x)$ [Fourier inversion formula]
        \item $\EE_{x\in \Gamma}\abs{f(x)}^2 = \sum_{\chi \in \hat{\Gamma}} \abs{\hat{f}(\chi)}^2  $ [Parseval identity]
    \end{enumerate}
\end{proposition}


\subsection{Approach of Alon--Pham}\label{subsec:AP-approach}

Independent sets in random Cayley graphs are \emph{clustered} in the sense that they are not usually disjoint from all other independent sets of the same size. Alon and Pham established this phenomenon in \cite{AH25-arxiv} and provided Theorem~\ref{thm:prev} as an application of this result. In this subsection, we explain their approach and present some main obstructions to extending their method to the hypergraph setting. 

We say that a family $\calV'$ is a \emph{cover} of a given family $\calV$ if each $V\in \calV$ contains an element $V' \in \calV'$ as a subset, referring to the terminology in \cite{AH25-arxiv}. As mentioned in the previous subsection, the main result of \cite{AH25-arxiv} is the existence of a cover for the collection $ \calS \defeq \{ A + A: A \subseteq \Gamma\}$ of small complexity consisting of sets that are dense in each $S$ in some sense. The existence of such cover allows us to obtain a sharper upper bound for the $p$ that a random Cayley (sum) graph generated by a $p$-random subset contains having a large independence number with high probability. Below, we include a brief demonstration of the idea that Alon and Pham~\cite{AH25-arxiv} used to find such covers.

For a subset $A \subseteq \Gamma$, we identify $A$ with its characteristic function.
The basic idea is to split into two cases depending on whether a subset $A\subseteq \Gamma$ has a small doubling constant. If $A$ has a small doubling constant, due to its rich additive structure, the convolution of the characteristic function, $A*A$ can be efficiently approximated by a sum of a small number of characters of $\Gamma$. This was proven by using discrete Fourier analysis. Then this approximation provides a sufficiently strong cover for sets of small doubling. Otherwise, if $A$ has a large doubling constant, the aforementioned approximation is inefficient. Hence, in this case, they claim that one can choose $A' \subseteq A$ such that its subset plays the role of a cover by sampling $A'$ from $A$ randomly.

There are several difficulties in extending the argument of Alon and Pham to $k$-fold sumsets and Cayley sum $k$-graphs. The first one is, in the step where sumsets are approximated by sums of characters, the case $k>2$ lacks an $L_k$ analogue of Parseval's identity, and there is also no asymmetric version of Parseval's identity. As a result, the approximation lemma \cite[Lemma~9]{AH25-arxiv} cannot be extended directly to our setting. We overcome this difficulty by establishing sufficiently strong inequalities for sums of products in the frequency domain, which work both in the asymmetric case and in the $L_k$ setting.

The second difficulty arises in the probabilistic argument used to obtain an efficient cover. In the work of Alon and Pham, the denseness of elements in the cover with respect to some $S \in \calS$ is shown by exploiting the fact that each of them contains a set that encodes part of the additive structure of $S = A+A$. The proof of the containment in the case of large doubling relies heavily on the fact that a Cayley sum graph generated by a single element is a matching. This structure guarantees independence of relevant random variables related to the additive structure of a randomly chosen set with respect to $A+A$. However, this is not the case in the higher uniformity setting. Instead, we discard a constant fraction of elements reflecting the additive structure of $A+A$. This does no harm to the denseness of the elements in the cover while significantly increasing the probability that the random sample becomes a valid element in the cover. To this end, we use several concentration inequalities such as Janson's inequality and Paley--Zygmund inequality. In this way, we obtain an efficient cover in both regimes, thereby extending the methodology of Alon and Pham to the hypergraph setting.

\section{Large sets are $k$-fold sumsets}\label{sec:lb}

Answering an analogue of Green's question in the higher uniformity setting, we prove the lower bound of Theorem~\ref{thm:main}. This generalizes a result from \cite{A07}, which says that any subset of $\ZZ_n$ of size at least $n-\Omega((n/ \log n)^{1/2})$ is a sumset for sufficiently large prime $n$. The result in \cite{A07} is shown by proving that if a set $X$ is sufficiently large, the probability that a sumset of a small random set being $X$ is nonzero.  We extend this method to a higher uniformity setting to prove an analogous result for $k$-fold sumsets.

\begin{proof}[Proof of the lower bound of Theorem~\ref{thm:main}]
     As the case when $k=2$ is proved by Alon in \cite{A07}, we assume $k$ to be at least $3$. Let $F \subseteq \Gamma$ be a set with $\abs{F} \le \frac{1}{4000e^3k^{2k+6}} \big(\frac{n}{(\log n)^{k}}\big)^{\frac{1}{k}} $. To prove the theorem, we find such $A$ from a collection of random sets whose probability of any $k$-fold sum of its elements in the set lying in $\Gamma \backslash F$ is maximized. Let $C$ be a random set constructed by choosing each element of $S \defeq \{ a : ka \notin F \} \subset \Gamma$ independently with probability $p=\big(\frac{20k^{k+3}\log n}{n^{k-1}}\big)^\frac{1}{k}$ and define $A$ as $C\backslash X$, while $X$ is defined as below. Note that $\abs{\{ a : ka \notin F \}}$ is equal to $n-\abs{F}$ as $x \mapsto kx$ is an injection by our assumption on the order of elements in $\Gamma$.
     \[
     X \defeq \{ a \in C : \exists a_1, \cdots ,a_{k-1} \in C \text{ such that }a+a_1+\cdots+a_{k-1}\in F \}. 
     \]
    
     By definition of $A$, we observe that no sum of $k$ elements in $A$ is in $F$. If $x \in A$, no $(k-1)$-tuple $(a_1, \cdots, a_{k-1})$ in $C^{k-1}$ satisfies $x+a_1+\cdots + a_{k-1}\in F$. This proves $kA \subseteq \Gamma \backslash F$. Now it suffices to show that $\Gamma \backslash F \subseteq kA$ holds with positive probability. To this end, we show that the probability of $x \notin kA$ for each $x \in \Gamma \backslash F$ is $o(1/n)$. For each $g \in \Gamma $, consider the Cayley sum hypergraph $\calH_g$ with vertex set $V(\calH_g) \defeq \Gamma $ with edge set $E(\calH_g)\defeq \{ \{x_1,\cdots, x_k\} : x_1+\cdots +x_k = g, x_1, \cdots , x_k \in \Gamma \}$. Note that we have at least $\binom{n}{k-1}\geq \frac{n^{k-1}}{(k-1)^{k-1}}$ many such solutions for the linear equation. An additional observation is that the number of edges intersecting $\Gamma \backslash S$ is negligible, and this guarantees that the hypergraph $\calH_g[S]$ has at least $k^{-k} n^{k-1}$ edges. 

    \begin{claim}\label{claim:lb-claim1}
        For each $g \in \Gamma \backslash F$, there exist more than $10k^3\log n$ pairs $(a_1^{(i)}, \cdots , a_k^{(i)}) \in C^k$ such that $a_1^{(i)}+\cdots +a_k^{(i)} = g $ with probability at least $1-o(1/n)$. 
    \end{claim}

    \begin{claimproof}[Proof of Claim~\ref{claim:lb-claim1}]
         Denote $\calH_g[S]$ as $H$. Consider the random variable $Y \defeq e(H[C])$, and for each $e\in E(H)$, let $C_e$ be an indicator variable such that $C_e = 1$ if $e\in E(H[C])$ and zero otherwise. We note that $Y= \sum_{e\in E(H)} C_e$. As the number of edges in a subgraph induced on a $p$-random set is monotone with respect to the number of edges in $\calH_g$, it suffices to bound the lower-tail probability for the case when the edge count is exactly $k^{-k}n^{k-1}$. We prove the claim by applying Lemma~\ref{lem:jan} to this setting. Denote $e \sim e'$ if $e$ and $e'$ are distinct edges that are not vertex-disjoint. Since each vertex is chosen independently at random with probability $p$, by the linearity of the expectation, we have
        \begin{equation}\label{eq:exp-Y}
            \EE[Y] = e(H) \cdot p^k = k^{-k} n^{k-1} p^k \ge 20 k^3 \log n.
        \end{equation}

        Also, $\Delta \defeq \sum_{e\sim e'}\EE [X_e X_{e'}]$ is computed as
        \allowdisplaybreaks
        \begin{align*}
            \Delta & \le 20k^3 \log n+ \sum_{i\in [k-2]} n^i \cdot \left( n^{k-1-i} \right)^2 \cdot p^{2k-i}
            \\ & = 20k^3 \log n+ \frac{1}{n^2} \sum_{i\in [k-2]} (np)^{2k-i}
            \\ & \le 20k^3 \log n+ (k-2)n^{2k-3}p^{2k-1}
            \\ & \le 20k^3 \log n+ (k-2)(20k^{k+3} \log n )^{(2k-1)/k} n^{-1/k}
            \\ &\leq 20k^3 \log n+ n^{-1/2k}.
        \end{align*}

        By Lemma~\ref{lem:jan2}, the probability of the event that the subhypergraph of $\calH_g$ induced on $X$ has less than $10k^3 \log n$ edges is bounded as
        \begin{align*}
            \PP [Y \le 10k^3 \log n] \le \exp \left(-\frac{(10k^3 \log n)^2}{80k^3 \log n} \right) \le \exp \left(-\frac{5}{4}k^3 \log n \right) \le o\left( \frac{1}{n^2} \right).
        \end{align*}
        
        Thus, there exist more than $10k^3\log n$ edges in $\calH_g [C]$ with probability at least $1-1/n^2$. This completes the proof.
    \end{claimproof}

As we want to bound the number of edges in $\calH_g[C]$ deleted by excluding elements from $X$, we prove the following probability estimates. 

    \begin{claim}\label{clm:3-statements}
        Fix $x\in \Gamma$. For each $g\in \Gamma \backslash F$, the following holds.
        \begin{enumerate}
            \item The probability of $C$ having more than $k^2$ distinct $(2k-2)$-tuples $(y_i, z_{i,1},\dots,z_{i,(k-1)}, $ $w_{i,1}, \dots,w_{i,(k-2)})$ such that $y_{i}+z_{i,1}+\cdots +z_{i,(k-1)}=g$ and $x+y_i+w_{i,1}+\cdots+w_{i,(k-2)}$ is in $F$ is $o(1/n^2)$.
            \item The probability of $C$ having more than $k^2$ distinct $(2k-2)$-tuples $(y_i, z_{i,1},\dots,z_{i,(k-1)},$ $ w_{i,1},\dots,w_{i,(k-2)})$ such that $x+z_{i,1}+\cdots +z_{i,(k-1)}=g$ and $y_i+x+w_{i,1}+\cdots+w_{i,(k-2)}$ is in $F$ is $o(1/n^2)$.
            \item The probability of $C$ having more than $k^2$ distinct $(2k-2)$-tuples $(y_i,y_i',x, z_{i,1},\dots,z_{i,(k-2)}, $ $ w_{i,1},\dots,w_{i,(k-2)})$ such that $y_i'+x+ z_{i,1}+\cdots +z_{i,(k-2)}=g$ and $y_i+x+w_{i,1}+\cdots+w_{i,(k-2)}$ is in $F$ is $o(1/n^2)$.
        \end{enumerate}
    \end{claim}
    
    \begin{claimproof}[Proof of Claim~\ref{clm:3-statements}] 
        Since the proof is completely symmetric across the cases, we omit the proofs of $(2)$ and $(3)$ and provide only $(1)$. Denote by $X_g$ the event such there exists more than $k^2$ distinct $(2k-2)$ tuples $(y_i, z_{i,1},\dots,z_{i,(k-1)}, $ $w_{i,1}, \dots,w_{i,(k-2)})$ such that
        \begin{equation}\label{eq:2k-2-tuple}
            y_{i}+z_{i,1}+\cdots +z_{i,(k-1)}=g \text{, } \ x+y_i+w_{i,1}+\cdots+w_{i,(k-2)} \in F
        \end{equation}

        holds. Then, denote by $X_g^{j}$ the event that there exists $k^2$ distinct such $(2k-2)$-tuples on $k^2(2k-2)-j$ distinct vertices. We bound the probability $\PP[X_g]$ by taking a union bound over $X_g^{j}$ for all $0\leq j \leq k^2(2k-2)$. 
        Note that the number of the $(2k-2)$-tuple of $\Gamma^{2k-2}$ that satisfies \eqref{eq:2k-2-tuple} and some of the $j'$ coordinates is fixed by some elements of $\Gamma$ is at most $n^{2k-4 - j'}|F|$. Thus, by double counting and greedily choosing the next tuple, the number of possible collections of size $k^2$ of $(2k-2)$-tuples that satisfy \eqref{eq:2k-2-tuple} and cover exactly $k^2(2k-2)-j$ vertices is at most $n^{k^2(2k-4)-j}|F|^{k^2}$. Then by the union bound, we have
        
        $$
            \PP[X_g] \le \sum_{j=0}^{k^2(2k-2)} \PP[X_g^{j}] \le \sum_{j=0}^{k^2(2k-2)}  n^{k^2(2k-4)-j}|F|^{k^2} p^{k^2(2k-2)-j}= o\left( \frac{1}{n^{k-1}}\right).
        $$

        As $k\ge 3$, this concludes the proof of the claim.

    \end{claimproof}

\begin{claim}\label{claim:lb-claim2}
    For each $g\in \Gamma$, the probability that there exists a $(2k-1)$-tuple $(x,y,z_1,\dots,z_{k-1},w_1,\dots $ $, w_{k-2})$ satisfying $x+y+w_1+\dots+w_{k-2} \in F$ and 
    $ y+z_1+\dots+z_{k-1}=g$ that intersects more than $3k^2(2k-1)\leq 6k^3$ other such tuples is $o(1/n)$.
\end{claim}

\begin{claimproof}[Proof of Claim~\ref{claim:lb-claim2}]
    Given $x', g' \in \Gamma$, denote by $Z_{x',g'}^{(1)}$ the event that $(1)$ in Claim~\ref{clm:3-statements} holds for the choice of $x=x'$ and $g=g'$. Also, denote the corresponding events for $(2)$ and $(3)$ of Claim~\ref{clm:3-statements} as $Z_{x',g'}^{(2)}$ and $Z_{x',g'}^{(3)}$, respectively. We apply the union bound twice. First, take a union bound over the events $Z_{x',g'}^{(i)}$ $(i \in [3])$. This proves that the probability of each element in a given tuple being contained in more than $3k^2$ other such tuples is $o(1/n^2)$. Then taking the union bound on the event $\vee_{i \in [3]}Z_{x',g'}^{(i)}$ for all $x'\in \Gamma$ proves the claimed bound.

\end{claimproof}

We now bound the number of disjoint such pairs.

    \begin{claim}\label{claim:lb-claim3}
       For each $g \in \Gamma \backslash F$, the probability of $C$ containing more than $s= \log n$ pairwise disjoint tuples $(x_i, y_i, z_{i,1},\dots,z_{i,(k-1)},w_{i,1},\dots ,w_{i,(k-2)})$ $(i \in [s])$ such that $y_i+z_{i,1}+\dots +z_{i, (k-1)}=g$ and $x_i+y_i+w_{i,1}+\dots +w_{i,(k-2)}\in F$ for each $i \in [s]$ is $o(1/n)$.
     \end{claim}

    \begin{claimproof}[Proof of Claim~\ref{claim:lb-claim3}]
    Note that it suffices to estimate the contribution from $(2k-1)$-tuples whose entries are all distinct, as the contribution from degenerate cases is significantly smaller for sufficiently large $n$. Therefore, the probability is bounded as
        \begin{align*}
            \binom{n^{2k-3}|F|}{s}p^{(2k-1)s}\leq \left( \frac{en^{2k-3}|F|p^{2k-1}}{s}\right)^s \leq \left( \frac{e(20k^{3+k})^{2-\frac{1}{k}}\log n}{4000e^3k^{2k+6}\log n} \right)^{\log n}=o\left(\frac{1}{n}\right).
        \end{align*}
    This completes the proof of the claim.
    \end{claimproof}

We are now ready to prove the theorem. Fix $g\in \Gamma$ and consider the $k$-uniform hypergraph $\calH_g$. Define the hypergraph $\calF_C$ as a hypergraph with vertex set $\Gamma$ and edge set 
$$
E(\calF_C)\defeq \{ (b_1, \dots , b_k) \in C^k:b_1+\cdots +b_k \in F\}.
$$
From the last claim, there exist at most $\log n$ pairwise disjoint pairs $(e_1, e_2)\in E(\calH_g) \times E(\calF_C)$ such that $e_1$ and $e_2$ share a vertex with probability $1-o(1/n)$. Since each such pair can be adjacent to at most $6k^3$ pairs with probability $1-o(1/n)$, applying the union bound proves that there exist at most $(6k^3+1)\log n$ such pairs $(e_1, e_2)$ with probability $1-o(1/n)$. Since we have more than $10k^3 \log n$ edges from the first claim, there exists at least one edge in $\calH_g$ after the deleting process. Then each $g\in \Gamma \backslash F $ is in $kA$ with probability at least $1-o(1/n)$. By applying the union bound on all $g\in \Gamma \backslash F$, we conclude that $\Gamma \backslash F \subseteq kA $ holds for some choice of $A$. This proves the theorem.
\end{proof}


\section{Efficient covers for $k$-fold sumsets}\label{sec:proof}

We generalize the covering lemma in \cite{AH25-arxiv} to the hypergraph setting to prove the existence of a large subset of $\Gamma$ that cannot be written as a $k$-fold sumset. This leads to the proof of our upper bound on $f_k(\Gamma)$.

\subsection{Finding a large cover}

We begin by finding a Fourier sparse pointwise approximation of $A_1*\cdots * A_k(x)$ to obtain a cover in the dense regime by establishing a $L_k$-inequality, which will play the role of the Parseval identity in the $2$-uniform case. 

\begin{definition}
    Let $\mathbf{A} = (A_1, \dots, A_k)$ be a $k$-tuple of subsets of $\Gamma$, where $|A_i| = \alpha_i |\Gamma|$. For each nonnegative integer $\ell$, we define $\mathbf{A}_{\ell}$ as
    $$
        \mathbf{A}_{\ell} \defeq \{ x : A_1*\cdots *A_k(x) \in (2^{-\ell-1}\alpha^{\frac{k-1}{k}}, 2^{-\ell}\alpha^{\frac{k-1}{k}}]\}.
    $$
\end{definition}

\begin{lemma}\label{lem:lp}
        Let $k\geq 2$ be a positive integer. Assume that $\alpha_i \in (0,1)$ for all $i \in [k]$ and $|A_i|=\alpha_i n$. Denote $\prod_{i=1}^k\alpha_i$ as $\alpha$. Then $\delta:=\sum_{\chi \in \hat{\Gamma}}{\prod_{j=1}^{k}|\hat{A_j}(\chi)|} \in [\alpha,\alpha^{(k-1)/k}]$.
\end{lemma}
\begin{proof}[Proof of Lemma~\ref{lem:lp}]
Let $A_1, \dots, A_k$ be subsets of $\Gamma$. As the largest size of the Fourier coefficient of $A_i (x)$ is equal to the density of $A_i$ in $\Gamma$, $\delta$ is bounded from above as
\begin{align*}
    \sum_{\chi \in \hat{\Gamma}}\prod_{i \in [k]}|\hat{A_i}(\chi)|  &\le \prod_{i \in [k]} \left( \sum_{\chi \in \hat{\Gamma}} |\hat{A_i}(\chi)|^k \right)^{1/k}  \\
    & \le \prod_{i \in [k]}\left( \sum_{\chi \in \hat{\Gamma}} \alpha_i^{k-2} |\hat{A}_i(\chi)|^2 \right)^{1/k}\\ 
    & \le \alpha^{\frac{k-2}{k}} \prod_{i \in [k]} \left( \sum_{\chi \in \hat{\Gamma}} |\hat{A}_i(\chi)|^2 \right)^{1/k}.
\end{align*}

The second inequality is justified by the fact that the Fourier transform of $A_i$ attains its maximum at the trivial character, which is exactly $\alpha_i$. Using Parseval's identity, we rewrite the product in the last expression using the $L_2$ norm of each $A_i(x)$. As $A_i(x)$ takes only $0$ or $1$ as its values, it is further simplified as the product of $L_1$ norms of $A_1(x), \dots, A_k(x)$. In other words,
\[
      \prod_{i \in [k]} \sum_{\chi \in \hat{\Gamma}} |\hat{A}_i(\chi)|^2 =  \prod_{i \in [k]} \EE_{x \in \Gamma} |A_i(x)|^2  = \prod_{i \in [k]} \alpha_i
    =\alpha. 
\]

Combining this with the previous inequality, we get $\delta \le \alpha^{(k-1)/k}$. It remains to prove the lower bound. Again, by Parseval's identity and the monotonicity of $L_p$ norms on probability measures, it holds that
\allowdisplaybreaks
\begin{align*} 
    \delta = \sum_{\chi\in \hat{\Gamma}}\prod_{i\in [k]}{|\hat{A_i}(\chi)|} & = \sum_{\chi \in \hat{\Gamma}}|\hat{A_1}(\chi)\cdots\hat{A_k}(\chi)| \\
    & \geq \big(\sum_{\chi \in \hat{\Gamma}}|(A_1*\cdots *A_k)^{\wedge}(\chi)|^2\big)^{\frac{1}{2}} \\
    & \geq ||(A_1*\cdots *A_k)^{\wedge}||_{l^2} \\
    & = ||A_1*\cdots *A_k||_{L^2} \\
    & \geq ||A_1*\cdots *A_k||_{L^1} \geq \prod_{i\in [k]} \alpha_i.
\end{align*}

Having the bound on both sides, we conclude the proof of the lemma.
\end{proof}

Next, we prove a pigeonholing lemma that proves the existence of a large set $\mathbf{A}_{\ell}$ whose elements $x\in \mathbf{A}_{\ell}$ are homogeneous in terms of the number of representations $(a_1, \dots, a_k) \in A_1 \times \cdots \times A_k$ such that $a_1+\cdots + a_k=x$.

\begin{lemma}\label{lem:pig}
    There is an absolute constant $c>0$ such that the following holds. Define $\alpha$ as $\prod_{i=1}^{k}\alpha_i$. Let $\mathbf{A}= (A_1, \dots, A_k)$ be a $k$-tuple of subsets of $\Gamma$ with each $A_i$ being of size $\alpha_i n$. If $|A_1+\cdots +A_k|\leq K\alpha^{1/k} n$, there exists $\ell\leq \log_2 K$ such that $|\mathbf{A}_\ell|\geq \frac{c2^\ell n\alpha^{1/k}}{(\ell+1)^2}$. 
\end{lemma}

\begin{proof}[Proof of Lemma~\ref{lem:pig}]
    For the sake of contradiction, assume that there does not exist an integer $1\le \ell \le \log_2 K$ such that $\mathbf{A}_\ell$ is popular. Note that $\EE _{x\in \Gamma}A_1*\cdots*A_k(x) = \alpha_1 \cdots \alpha_k$. Then with an appropriate choice of $c$, we have 
    \begin{align*}
        \sum_{x\in \Gamma}A_1*\cdots *A_k(x) & = n \prod_{i\in [k]} \alpha_i\\ 
        &< \sum_{\ell\leq \log K} 2^{-\ell}\alpha^{(k-1)/k}\cdot \frac{c2^{\ell}\alpha^{1/k}n}{(\ell+1)^2} + 2^{-(\log K +1)} \alpha^{(k-1)/k}K\alpha^{1/k}n \\ 
        &< n \prod_{i\in [k]} \alpha_i,
    \end{align*}
    a contradiction. This completes the proof.
\end{proof}

Given a $k$-tuple of sets $\mathbf{A} = (A_1, \dots, A_k)$, denote the smallest integer $\ell$ such that $|\mathbf{A}_\ell|\geq \frac{c2^\ell n\alpha^{1/k}}{(\ell+1)^2}$ as $\ell(\mathbf{A})$. 

\begin{lemma}\label{lem:FSA}
Let $A_1, \dots , A_k$ be subsets of $\Gamma$ with size $|A_i|=\alpha_in$ for each $i\in [k]$, and let $\eta \in (0,1)$. Define $\delta$ as $\sum_{\chi \in \hat{\Gamma}}{\prod_{j=1}^{k}|\hat{A_j}(\chi)|} $ and $\alpha$ as $\prod_{i=1}^{k}\alpha_i$. Then there exists a collection of $m=5\eta^{-2}\log n$ characters $\{ \chi_i \}_{i=1}^m$ such that the following holds.
    $$\max_{x \in \Gamma}\biggr\{\biggr\lvert A_1*\cdots*A_k(x) - \frac{\delta}{m}\sum_{i=1}^{m}{\frac{\prod_{j=1}^k\hat{A_j}(\chi_i)}{|\prod_{j=1}^k\hat{A_j}(\chi_i)|}}\chi_i(-x)\biggr\rvert\biggr\}\leq \eta\delta .$$
\end{lemma}
\begin{proof}[Proof of Lemma~\ref{lem:FSA}]
    First, we pick $m$ characters $\chi_1,\cdots ,\chi_m \in \hat{\Gamma}$ independently at random with probability $\PP [\chi_i=\chi]=\frac{\prod_{j=1}^{k}|\hat{A_j}(\chi)|}{\delta}$ for each $i\in[k]$. We bound the probability that the approximation fails to achieve the desired level of accuracy for each $x \in \Gamma$. Note that the value of $A_1*\cdots *A_k$ at each point can be written as
    
    $$
    \EE \Big[\frac{\delta}{m}\sum_{i=1}^{m}\frac{\prod_{j=1}^k\hat{A_j}(\chi_i)}{|\prod_{j=1}^k\hat{A_j}(\chi_i)|}\chi_i(-x)\Big]=A_1*\cdots*A_k(x),
    $$

    from the Fourier inversion formula. By Hoeffding's bound, the probability of an approximation being at least $t\delta m^{-1/2}$ far from the exact value of $A_1 * \cdots * A_k (x)$ is bounded as below for any positive real $t$.
    
    $$
    \PP \Big[|A_1*\cdots*A_k(x)-\frac{\delta}{m}\sum_{i=1}^m\frac{\prod_{j=1}^{k}\hat{A_j}(\chi_i)}{|\prod_{j=1}^{k}\hat{A_j}(\chi_i)|}\chi_i(-x)|>t\frac{\delta}{\sqrt{m}}\Big]\leq e^{-\frac{t^2}{4}}.
    $$

    Now take $t=\eta \sqrt{m}$. Then, 

    \begin{align*}
    \PP \Big[|A_1*\cdots *A_k(x)-\frac{\delta}{m}\sum_{i=1}^m\frac{\prod_{j=1}^{k}\hat{A_j}(\chi_i)}{|\prod_{j=1}^{k}\hat{A_j}(\chi_i)|}\chi_i(-x)|>\eta\delta \Big] & \leq e^{-\frac{m \eta^2 }{4}} \\
    & \leq n^{-\frac{5}{4}}.
    \end{align*}
    
    As the probability that the approximation fails at each $x$ is $o(1/n)$, taking a union bound over all possible choices of $x$ proves that there exists a choice of $\chi_1,\cdots ,\chi_m$ satisfying the approximation property. This completes the proof.
\end{proof}

We now prove a higher uniformity analogue of the covering lemma using the previous lemmas that is efficient whenever $\mathbf{A} =(A_1, \dots, A_k)$ has a small doubling constant.

\begin{lemma}\label{lem:cov1}
    There are absolute constants $C, c >0$ such that the following holds. Define $\alpha$ as $\prod_{i\in [k]}\alpha_i$. Given an abelian group $\Gamma$ of size $n$ and $\alpha_i \in (0,1)$, there exists a collection $\calF_\ell$ of subsets of $\Gamma$ such that 
    
    $$
        |\calF_\ell|\leq \exp (C2^{2\ell}(\log n)^2),$$ and $$\min_{F\in \calF_\ell}|F| \geq c2^\ell \alpha^{1/k}n\frac{1}{(\ell+1)^2},
    $$
    and the following property holds.
    
    For any $k$-tuple of subsets $\mathbf{A} = (A_1, \dots, A_k) \subseteq \Gamma$ such that $|A_i|=\alpha_i n$ for each $i \in [k]$ and $|A_1+\cdots +A_k| \leq K\alpha^{1/k}n$, there exists $F\in \calF_\ell$ that satisfies 
    $$
        \mathbf{A}_\ell \subseteq F \subseteq A_1+ \cdots + A_k,
    $$ 
    where $\ell \defeq \ell(\mathbf{A})$. 
\end{lemma}

\begin{proof}[Proof of Lemma~\ref{lem:cov1}]
     Let $\epsilon =\frac{c}{n}$. Take an $\frac{\epsilon}{32}$-net of the unit circle $\{x\in \mathbb{C} : |x|=1 \}$ of size $O(\epsilon^{-1})$, and denote it as $\calU$. For each integer $\ell$ and fixed $\delta$, define the collection of functions that can be represented as a linear combination of small number of characters as
     
     \[
     \calG_\ell \defeq \{ \frac{\delta}{m}\sum_{\chi \in S} t_\chi \chi(-x) :S \subseteq \hat{\Gamma}, |S|= 2^{2\ell+8}\log n, t_\chi \in \calU \}.
     \]
    
    Note that $S$ is a multiset with elements in $\hat{\Gamma}$. Define $\delta_i \defeq \alpha2^i$, and consider the dyadic partition of the interval $[\alpha, \alpha^{1/k}]$ as $ \{ [\delta_i, \delta_{i+1}) : i = 0, \cdots, \lfloor \log \alpha^{-1/k} \rfloor \}$. 

    Now consider the normalized family
    \[
        \calG^{(i)}_{\ell} \defeq \big\{ \frac{\delta_i}{m}\sum_{\chi \in S}t_\chi\chi(-x) : t_{\chi}\in \calU, S\subseteq \hat{\Gamma}, |S| \leq 2^{2\ell+8}\log n \big\}.
    \]

    We claim that we can find the desired cover by exploiting the approximation property of the family $\bigcup_\ell(\cup_i\calG^{(i)}_\ell)$ using the result in Lemma~\ref{lem:FSA}. Define $\calF_{\ell}^{(i)}$ as the collection of sets of the form 
    \[
    \{x : \abs{f(x)}>2^{-\ell-3}\delta_i \}
    \]
    for each $f \in \calG_\ell^{(i)} $ and $i \leq \lfloor \log \alpha^{-1/k}\rfloor $. It suffices to show that $\calF_{\ell}\defeq \cup_i\calF^{(i)}_\ell $ forms a cover with the desired properties.

    \begin{claim}\label{clm:Fell-upper}
        $\abs{\calF^{(i)}_\ell}\leq \exp (C2^{2\ell}(\log n)^2)$. 
    \end{claim}
    \begin{claimproof}[Proof of Claim~\ref{clm:Fell-upper}]
        There are $n^{O(2^{2\ell+8}\log n)}$ choices for $t$ and $\binom{n}{2^{2\ell+8}\log n}$ choices for $S,$ which implies that the size of $\calG^{(i)}_\ell$, whose size is at most that of $\calF^{(i)}_\ell$, is of $\exp (O(2^{2\ell+8}(\log n)^2))$. This completes the proof of the claim.
    \end{claimproof}

    Let $\mathbf{A} = (A_1, \dots, A_k)$, a $k$-tuple of subsets of $\Gamma$ such that $\abs{A_i} = \alpha_i n$ for each $i \in [k]$ and $\abs{A_1 + \cdots + A_k } \le K\alpha^{1/k}n$, be given. 
    
    \begin{claim}\label{clm:Fell-cover}
        There exists $F \in \calF_\ell$ such that $\mathbf{A}_\ell \subseteq F \subseteq A_1+\cdots +A_k$.
    \end{claim}

    \begin{claimproof}[Proof of Claim~\ref{clm:Fell-cover}]
        First, we find a set that is contained in $A_1+\cdots +A_k$ by finding a function that approximates $A_1*\cdots*A_k$ properly. Then we prove that the chosen set contains a structured set $\mathbf{A}_\ell$ inside it. Let $i \leq \lfloor \log_2\alpha^{-1/k} \rfloor$ be the largest integer such that $\delta \geq \delta_{i}$. By the approximation lemma, there exists a set of characters $S \subset \hat{\Gamma}$ of size $5\eta^{-2}\log n$ such that 
        \[
        \bigg\lvert A_1*\cdots*A_k(x)-\frac{\delta}{m}\sum_{\chi \in S}t_\chi \chi(-x)\bigg\rvert\leq 2\eta \delta.
        \]

        Take $\eta = 2^{-\ell-4}$. Then by our choice of $\delta_i$, for all $x \in \mathbf{A}_\ell$,
        \allowdisplaybreaks
        \begin{align*}
           \bigg\lvert A_1*\cdots* A_k(x) & -\frac{\delta_i}{m}\sum_{\chi \in S}t_\chi \chi(-x) \bigg\rvert \\
            & \leq \bigg\lvert A_1*\cdots*A_k(x) \cdot \frac{\delta_i}{\delta}-\frac{\delta_i}{m}\sum_{\chi \in S}t_\chi \chi(-x)\bigg\rvert + \bigg\lvert(1-\frac{\delta_i}{\delta})A_1*\cdots *A_k(x)\bigg\rvert \\
            & \leq \frac{\delta_i}{\delta} \bigg\lvert A_1*\cdots*A_k(x)-\frac{\delta}{m}\sum_{\chi\in S}t_\chi \chi(-x)\bigg\rvert + \frac{1}{2}\bigg\lvert A_1*\cdots*A_k(x)\bigg\rvert \\
            & \leq \frac{\delta_i}{\delta}\cdot 2\eta \delta + \frac{1}{2} \bigg\lvert A_1*\cdots*A_k(x)\bigg\rvert  \\
            & = 2 \eta \delta_i + \frac{1}{2} \bigg\lvert A_1*\cdots*A_k(x)\bigg\rvert.
        \end{align*} 

    Then from positivity of $A_1*\cdots*A_k(x)$ and $\delta_i \le \delta \le \alpha^{(k-1)/k} $,
    $$
        \frac{1}{2} \cdot A_1*\cdots*A_k(x)-2\eta \alpha^{(k-1)/k} \leq \abs{\frac{\delta_i}{m}\sum_{\chi \in S} t_\chi \chi (-x)} \leq \frac{3}{2} \cdot A_1*\cdots*A_k(x)+2 \eta \alpha^{(k-1)/k}.  
    $$

    Also, from $x \in \mathbf{A}_\ell$ and the definition of $ \mathbf{A}_\ell$,

    $$2^{-\ell-3}\alpha^{(k-1)/k} < \abs{\frac{\delta_i}{m}\sum_{\chi \in S} t_\chi \chi (-x)} < 2^{-\ell+2}\alpha^{(k-1)/k} .$$
    
    Since $\delta \leq \alpha^{(k-1)/k}$,  we observe that any $x$ in $\mathbf{A}_\ell$ is also in $F$. This proves the claim.
    
    \end{claimproof}

Since the minimum size condition for each $F \in \calF_\ell$ is automatically satisfied from the choice of $\ell$, the set $\calF_{\ell}$ is the desired family. This completes the proof.

\end{proof}

We now construct a large efficient cover for the case when $\mathbf{A} =(A_1, \dots, A_k)$ has a large doubling constant. A cover that contains full $\mathbf{A}_\ell$ for all families of size $\alpha n$ subsets $A_1, \dots, A_k$ is hopelessly large in very sparse regime, so we construct a cover that contains a large portion of each $\mathbf{A}_\ell$ instead. The following lemma bounds the probability of each element being in the $k$-fold sumset of random subsets of $A_1, \dots , A_k$.

\begin{lemma}\label{lem:prob}
    Let $A_1, \dots, A_k$ be subsets of $\Gamma$ of size $\alpha n$, and let each $B_1, \dots, B_k$ be $p$-random subsets of $A_1, \dots, A_k$, respectively. Then there exists an absolute constant $\gamma$ depending only on $k$ such that the following holds. Given $x \in \Gamma$, define the random variable $X_x \defeq n^{k-1}B_1*\cdots *B_k(x)$. Then if $(\alpha n)^{-1} \le p \le (\alpha n)^{-(k-2)/(k-1)} $, the following holds.
    \begin{enumerate}
        \item[$(a)$] $\PP [X_x= 0] \leq \exp{(-\gamma \EE[X_x])}$,
        \item[$(b)$] $ \PP[X_x > 0 ] \geq \gamma \min \{\EE[X_x], 1\}$. 
    \end{enumerate}
\end{lemma}

\begin{proof}[Proof of Lemma~\ref{lem:prob}] 
    Given $x\in \Gamma$ and $A_1, \dots, A_k \subseteq \Gamma$, we denote by $r_x$ the number of configurations $(a_1, \dots, a_k) \in A_1 \times \cdots \times A_k$ such that $a_1+\dots+a_k =x$ hold. For each pair of $k$-tuples $\mathbf{a} = (a_1, \dots, a_k)$ and $\mathbf{b} =(b_1, \dots, b_k)$ agreeing at $j$ coordinates, the probability of both pair being contained in $B_1\times \cdots \times B_k$ is $p^{2k-j}$. There are also at most $\binom{k}{j}(\alpha n)^{k-j-1}r_x$ many such pairs as given $\mathbf{a}$ and $j$ coordinates of $\mathbf{a}$, we have no more than $(\alpha n)^{k-j-1}$ choices for the remaining coordinates in $\mathbf{b}$. Let $E_{\mathbf{a}}$ denote the event that all components of $\mathbf{a}$ are in $B_1\times \cdots \times B_k$ and $E_{\mathbf{b}}$ denote the event that all components of $\mathbf{b}$ are in $B_1\times \cdots \times B_k$. Say that $\mathbf{a}\sim \mathbf{b}$ if $\mathbf{a}$ and $\mathbf{b}$ coincides on some index set $I \subseteq [k]$. Similarly, say $\mathbf{a} \perp \mathbf{b}$ if the set of components of $\mathbf{a}$ and $\mathbf{b}$ are disjoint. Then the expectation $\mu_x$ and the dependency term $\Delta_x$ for each $x$ are computed as below.
    \[
    \mu_x \defeq \EE [X_x] = p^{k} r_x
    \]
    \[
    \Delta [X_x] = \sum_{\mathbf{a}\sim \mathbf{b}}\PP[E_{\mathbf{a}} \wedge E_{\mathbf{b}}] \le \sum_{j=1}^{k-2}\binom{k}{j}p^{2k-j}(\alpha n)^{k-j-1}r_x
    \]

    As $(\alpha n)^{-1} \le p \le (\alpha n)^{-(k-2)/(k-1)} $, there exists a constant $\gamma'$ such that $\Delta_x \le \gamma' \mu_x$. By Lemma~\ref{lem:jan}, we get the desired bound as follows, with $\gamma'$ being a constant depending only on $k$.

    \[
    \PP [X_x= 0] \le \exp{ \left(-\frac{p^{2k} \{r_x\}^2}{(1+ \gamma')p^k r_x}\right)} \le \exp{\left(-\frac{\EE[X_x]}{1+ \gamma'} \right)}.
    \]

    The second inequality is proved by using the Paley--Zygmund inequality. Denote by $T_x$ the set of all $k$-tuples in $A_1 \times \cdots \times A_k$ that satisfies $a_1+\cdots +a_k=x$. Recall that $\EE [X_x]= \mu_x$, and
    
    \begin{align*}
        \EE [X_x^2] & = \sum_{\mathbf{t} \in T_x} \PP [E_\mathbf{t}^2] +  \sum_{\mathbf{t}_1 \sim \mathbf{t}_2 }\PP [E_{\mathbf{t}_1} \wedge E_{\mathbf{t}_2}] + \sum_{\mathbf{t}_1 \perp \mathbf{t}_2 }\PP [E_{\mathbf{t}_1} \wedge E_{\mathbf{t}_2}] \\
    & \leq \mu_x + \gamma' \mu_x +  \mu_x^2
    \end{align*}
    
    holds with $\gamma_2$ being a nonzero constant depending only on $k$. Note that the first term of the sum after the first inequality can be simply interpreted as $\mu_x$ as $E_{\mathbf{t}}$ is an indicator variable for any $\mathbf{t}\in T_x$. 
    
    Then by Lemma~\ref{lem:pzy}, 
    \[
    \PP[X_x > 0 ] \ge \frac{\EE [X_x]^2}{\EE [X_x^2]} \ge \frac{\mu^2_x}{\mu_x^2 + (\gamma'+1) \mu_x} \ge \frac{\min \{ 1,  \mu_x\}}{2+\gamma'}.
    \]
    
    The final inequality is obtained by considering the cases when $0 \le \mu_x \le 1$ and $\mu_x > 1$ separately. Then we choose a universal constant $\gamma$ by choosing the minimum of both constants. This completes the proof.
\end{proof}

\begin{lemma}\label{lem:cov2}
    There are absolute constants $C, c>0$ depending solely on $k$ such that the following holds. Given an abelian group $\Gamma$ of size $n$ and $n^{-1}(\log n )^{10k} \le \alpha \le 1 $, there exists a collection $\calH_\ell$ of subsets of $\Gamma $ that satisfies both
    \begin{equation}\label{eq:size1}
        \abs{\calH_\ell} \le \exp (C \cdot \min \{ (2^\ell \alpha n)^{1/k}, (\alpha n)^{1/(k-1)}\} (\log n)^{(k+1)/k})
    \end{equation}
    
    and 
    \begin{equation}\label{eq:size2}
        \min_{H\in \calH_\ell} \abs{H} \ge c \cdot \min\{ 2^\ell \alpha n, (\alpha n)^{k/(k-1)} \}\frac{1}{(\ell+1)^2},
    \end{equation}
    and such that the following property holds. 
    For any $k$-tuple $\mathbf{A} = (A_1, \dots, A_k)$ of size $\alpha n$ subsets of $\Gamma$ such that $\abs{A_1+\cdots + A_k}\le K \alpha^{1/k}n $ for some $\alpha \in [(\log n)^{10k}n^{-1}, 1]  $ holds, there exists $H \in \calH_\ell$ that satisfies 
    $$
    H \subseteq A_1 + \cdots + A_k,
    $$
    where $\ell \defeq \ell(\mathbf{A})$.
\end{lemma}

\begin{proof}[Proof of Lemma~\ref{lem:cov2}]
    We divide into two cases. Recall that given a $k$-tuple of subsets of $\Gamma$, $\mathbf{A}= (A_1, \dots, A_k)$, $\mathbf{A}_\ell$ is defined as the collection of elements $x\in \Gamma$ such that 
    \[
    A_1 * \cdots * A_k (x) \in (2^{-\ell-1}\alpha^{(k-1)/k}, 2^{-\ell}\alpha^{(k-1)/k}]
    \]
    holds. Let $\gamma$ be the same constant in the statement of Lemma~\ref{lem:prob}. Whenever $\ell \leq \log (\gamma \frac{(\alpha n)^{1/(k-1)}}{\log n}) $, we sample a small collection of subsets in each $A_i$ to contain the whole $\mathbf{A}_\ell$. Otherwise, we rely on Lemma~\ref{lem:prob} to justify the existence of a small cover whose existence of an element that contains a large part of each $\mathbf{A}_\ell$ is guaranteed. In the following discussion, let $c_0$ be a constant to be chosen later. Let $C_1, C_2$ and $c_1, c_2$ be constants to be chosen in each of the cases. After choosing the constants in both cases, we choose universal constants $C$ and $c$ as $\max \{ C_1, C_2\}$ and $\min \{ c_1, c_2\}$.

\begin{case}
    $\ell \leq \log \left(c_0 \frac{(\alpha n)^{1/(k-1)}}{\log n}\right) $.
\end{case}
    In this case, our choice of $\ell$ implies that $2^\ell \alpha n$ is at most a constant factor of $(\alpha n)^{k/(k-1)}$. Let $\zeta>0$ be a constant depending only on $k$ such that $\zeta \ge (20/\gamma)^{1/k}$. Define the family $\calH_\ell$ as
    \[
    \overline{\calH}_\ell \defeq \left\{ B_1+ \cdots + B_k : \forall i \in [k], \ B_i \subseteq \Gamma, \ \frac{p\alpha n}{2} \leq \abs{B_i} \leq 2p\alpha n \right\}
    \]
    with the choice of $p=\zeta (2^\ell \alpha n \log n)^{1/k} (\alpha n)^{-1}$. We claim that an appropriate subfamily of $\overline{\calH_\ell}$ is the desired family. We first prove that the family is of size at most $\exp (C_1(2^\ell \alpha n)^{1/k}(\log n)^{(k+1)/k})$, and choose $C > C_1$ after dealing with both cases. Since we will choose $\calH_{\ell}$ as a subfamily of $\overline{\calH}_{\ell}$, it suffices to bound the size of $\overline{\calH}_{\ell}$. Then there exists an absolute constant $C_1 >0$ such that
    \begin{equation}\label{eq:bar-calH-size1}
        \abs{\overline{\calH}_\ell} \le \prod_{j=1}^k \sum_{ i=p\alpha n/2 }^{2p\alpha n}\binom{n}{i} \le  2p\alpha n\cdot n^{2p\alpha n}  \le \exp (C_1(2^\ell \alpha n)^{1/k}(\log n)^{(k+1)/k}),
    \end{equation}
    This will later be used to prove that the family $\calH_\ell$, which is to be defined as a subfamily of $\overline{\calH_\ell}$, is not too large.

    We now fix $k$ subsets $A_1, \dots , A_k$ of $\Gamma$ of size $\alpha n$. For each $i \in [k]$, let $A_i'$ be a $p$-random subset of $A_i$ with $p \ge \zeta (2^\ell \alpha n \log n)^{1/k}(\alpha n)^{-1} $. For each $x \in \mathbf{A}_\ell$, define the number of tuples in $A_1 \times \cdots \times A_k$ such that $a_1+\cdots + a_k=x$ as $R_x$. Also, denote the random variable defined by the number of $k$-tuples in $A_1'\times \cdots \times A_k'$ such that $a_1'+\cdots + a_k' =x  $ as $X_x$. By the definition of $\mathbf{A}_\ell$, each element $x \in \mathbf{A}_\ell$ satisfies $ R_x \ge 2^{-\ell-1}(\alpha n)^{k-1}$. Then the expectation $\EE [X_x]$ is $p^k R_x$. Then we have
        \[
        p^k R_x  \geq \zeta^k (2^\ell \alpha n \log n)\frac{1}{(\alpha n)^k} \cdot 2^{-\ell-1} (\alpha n)^{k-1} \geq \frac{\zeta^k}{2} \log n.
        \]
    Recall that $\zeta$ is chosen to be $\zeta \ge (20/\gamma)^{1/k}$. Then by Lemma~\ref{lem:prob}, the probability of each $x \in \mathbf{A}_\ell$ not being in $A_1'+\cdots + A_k'$ satisfies 
        
        \[
        \PP[x\notin A_1'+\cdots +A_k'] \leq \exp{(-\gamma \EE[X_x])} \leq \exp{(-10 \log n)} =o\left(\frac{1}{n^8}\right).
        \]

    Hence, by the union bound, there exists a choice of $A_1', \dots , A_k'$ such that
    \begin{equation}\label{eq:contain}
        \mathbf{A}_\ell \subseteq A_1'+\cdots +A_k'
    \end{equation}

    holds with probability at least $1- o(n^{-7})$. To prove the existence of such $A_1'+\cdots +A_k'$ in $\overline{\calH}_\ell$, it remains to show that their size does not differ from $p\alpha n $ by more than factor $2$ with high probability. Recall that $\alpha n$ was chosen to be at least $(\log n)^{10k}$ and $2^\ell \ge 1$. By the Chernoff bound, there exists some absolute constant $\eta$ depending only on $k$ such that 
        \[
        \PP [ \exists i \in [k] \text{ s.t. } \abs{A_i'} > 2p\alpha n] \leq k \cdot  \exp \left( -\frac{\eta}{4}(2^\ell \alpha n \log n)^{1/k} \right) = o\left(\frac{1}{n^3}\right).
        \]

    Similarly, we bound the lower tail probability as
        \[
        \PP [ \exists i \in [k] \text{ s.t. } \abs{A_i'} < \frac{p\alpha n}{2}] \leq k \cdot \exp \left( -\frac{\eta}{4}(2^\ell \alpha n \log n)^{1/k} \right) = o\left(\frac{1}{n^3}\right).
        \]
    The last inequality follows from our restriction on $\alpha n$. Combining this with the two bounds we already have shown, we conclude that there is a subcollection of $\overline{\calH}_\ell$ that has an element that contains $\mathbf{A}_\ell$ for any choice of $A_1, \dots, A_k$. Take only the sets $B_1+ \cdots + B_k \in \overline{\calH}_\ell$ that contains $\mathbf{A}_\ell$ for some $\mathbf{A} = (A_1, \dots, A_k)$ whose $\ell(\mathbf{A})$ is equal to $\ell$. As $\calH_\ell$ is chosen to be a subfamily of $\overline{\calH}_\ell$, it is shown that the family is of size at most $\exp(C_1 (2^\ell \alpha n)^{1/k} (\log n)^{(k+1)/k})$. From \eqref{eq:contain}, we get our desired family $\calH_\ell$ as each $\mathbf{A}_\ell$ are of size at least $c2^{\ell}n \alpha^{1/k}/(\ell+1)^2$, and this is at least $c2^\ell \alpha n / (\ell+1)^2$ as $\alpha \le 1$.
    
    \begin{case}
    $\ell > \log (c_0 \frac{(\alpha n)^{1/(k-1)}}{\log n}) $.
    \end{case}

    Let $\overline{\calH}_\ell$ be defined as in the previous case with the choice of $p= c_0 (\alpha n)^{-(k-2)/(k-1)}  $. We define the family $\calH_\ell$ similarly by taking all sets that emerge as $k$-fold sumsets $B_1+\cdots +B_k$ containing a large portion of $\mathbf{A}_\ell$ for some $\mathbf{A}= (A_1 , \dots , A_k)$. To this end, we first verify that there exist absolute constants $C_2$ and $c_2$ such that the family is of size at most $\exp(C_2 (\alpha n)^{1/(k-1)} \log n )$ and each element in the family is of size at least $c_2 \cdot \min\{ 2^\ell \alpha n, (\alpha n)^{k/(k-1)} \}/(\ell+1)^2$, respectively. To check that the family is of small size, it suffices to show that the same bound holds for $\overline{\calH}_\ell$. Observe that there exists an absolute constant $C_2>0$ depending only on $k$ such that 

    \begin{equation}\label{eq:bar-calH-size2}
    \abs{\overline{\calH}_\ell} \le 
    \prod_{j=1}^k \sum_{ i=p\alpha n/2 }^{2p\alpha n}\binom{n}{i} \le n^{2 (\alpha n)^{1/(k-1)}} \le \exp(C_2 (\alpha n)^{1/(k-1)} \log n )
    \end{equation}
    
    holds. In this regime, $(\alpha n)^{1/(k-1)}$ is at most a constant factor of $(2^\ell \alpha n)^{1/k}(\log n)^{1/k}$, so the desired bound is implied by the inequality above.

    We now claim that given $\mathbf{A} =(A_1, \dots, A_k)$, a $k$-tuple of size $\alpha n$ subsets of $\Gamma$ such that $\ell (\mathbf{A})=\ell$, there exists an element of $\overline{\calH}_\ell$ that contains a large portion of $\mathbf{A}_\ell$. For each $i\in [k]$, let $A_i'$ be a $p$-random subset of $A_i$. For each $x \in \mathbf{A}_\ell$, let $R_x$ and $X_x$ be defined as in the previous case. Then by (b) of Lemma~\ref{lem:prob},
    \allowdisplaybreaks    
    \begin{align*}
        \PP [x \in B_1+\cdots +B_k] & \geq \gamma \min \{ 1, \EE[X_x]\} \\ & \geq \gamma \min \{ 1, p^k R_x\} 
        \\ & \geq \gamma \min \{ 1, c_0^k 2^{-\ell}(\alpha n)^{1/(k-1)}\} \\ & \geq c_k \min \{1, 2^{-\ell}(\alpha n)^{1/(k-1)}\}
    \end{align*}

    holds with $c_k \defeq \gamma c_0^k$. This implies  
        \[
        \EE [\abs{\mathbf{A}_\ell \cap (A_1'+\cdots +A_k')}] \geq  c_k \min \{1, 2^{-\ell}(\alpha n)^{1/(k-1)}\}\abs{\mathbf{A}_\ell}.
        \]

    We prove that it is enough to choose sets of size that differs by at most factor $2$ from $p\alpha n$. By Chernoff bound, there exists some absolute constant $\eta'$ depending only on $k$ such that
    
        \[
        \PP [\exists i \in [k] \text{ s.t. } \abs{A_i'} > 2p\alpha n] \leq k \cdot \exp \left( -\frac{\eta'}{4}( c_0(\alpha n)^{1/(k-1)}) \right) = o\left(\frac{1}{n^3}\right)
        \]
        
    holds. As we did in the previous claim, we also bound the lower tail probability as
    
        \[
        \PP [\exists i \in [k] \text{ s.t. } \abs{A_i'} < 2^{-1}p\alpha n] \leq k \cdot \exp \left( -\frac{\eta'}{4}( c_0(\alpha n)^{1/(k-1)}) \right) = o\left(\frac{1}{n^3}\right).
        \]
    Again, the last inequality follows from our choice of $\alpha n$, which guarantees $ (\alpha n)^{1/(k-1)}> (\log n)^{9}$. Combining the bounds altogether, we now proceed to show that for any choice of sets $A_1, \dots, A_k \subseteq \Gamma$, there exists $H \in \overline{\calH}_\ell$ that satisfies 
    \begin{equation}\label{eq:largeint}
        \abs{H \cap \mathbf{A}_\ell} \ge c_k \min \{ 1, 2^{-\ell} (\alpha n)^{1/(k-1)}\}.
    \end{equation}
    
    Formally, we prove the fact as below. Denote by $\mathcal{E}$ the event such that $A_i' \subseteq A_i$ satisfies $2^{-1}p \alpha n \le \abs{A_i'} \le 2p\alpha n$ for each $i\in [k]$. Then by the union bound, we have $\PP [\overline{\mathcal{E}}] = o (1/n^3)$. From this, we obtain
    \begin{align*}
        \EE [\abs{H \cap \mathbf{A}_\ell}] & =  \EE[\abs{H \cap \mathbf{A}_\ell} | \mathcal{E}] \cdot \PP[\mathcal{E}] +\EE[\abs{H \cap \mathbf{A}_\ell} | \overline{\mathcal{E}}]  \cdot \PP[\overline{\mathcal{E}}] \\
        &  \le \EE[\abs{H \cap \mathbf{A}_\ell} | \mathcal{E}] \cdot \PP[\mathcal{E}] + o \left( \frac{1}{n^2} \right), 
    \end{align*}
    where the second inequality holds as $|H \cap \mathbf{A}_{\ell}| \leq |\mathbf{A}_{\ell}| \leq n$. Observe that $\abs{\mathbf{A}_\ell}\ge 2^{\ell} \alpha n/(\ell+1)^2 $ implies $2^{-\ell }(\alpha n)^{1/(k-1)} \abs{\mathbf{A}_\ell} \ge (\alpha n)^{k/(k-1)}/ (\ell+1)^2$, and we have $\ell \le \log K \le 2\log n$ by the trivial bound on doubling constants. Therefore, for sufficiently large $n$,
    \begin{align*}
    \EE[\abs{H \cap \mathbf{A}_\ell} | \mathcal{E}]  & \ge \frac{c_k}{\PP[\mathcal{E}]} \min \{ 1, 2^{-\ell}(\alpha n)^{1/(k-1)}\}\abs{\mathbf{A}_\ell} - 1 \\
    & \ge \frac{c_k}{2} \min \{ 1, 2^{-\ell}(\alpha n)^{1/(k-1)}\}\abs{\mathbf{A}_\ell}.
    \end{align*}
    
    This implies that there exists $H \in \overline{\calH}_\ell$ that contains at least $c_k \min \{ 1, 2^{-\ell}(\alpha n)^{1/(k-1)}\}$ portion of $\mathbf{A}_\ell$. Choose $\calH_\ell$ as the collection of all elements $H $ in $\overline{\calH}_\ell$ such that there exists $\mathbf{A} = (A_1, \dots, A_k)$ whose $\ell (\mathbf{A})$ is equal to $ \ell$ and satisfies \eqref{eq:largeint}.
     As $\calH_\ell$ is defined to be a subfamily of $\overline{\calH}_\ell$, it is shown that the family is of size at most $\exp (C_1 (2^\ell \alpha n)^{1/k}(\log n)^{(k+1)/k})$. Moreover, since $\abs{\mathbf{A}_\ell} \ge c2^\ell \alpha n /(\ell+1)^2$, this guarantees that the elements of $\calH_\ell$ are of size at least $c_2 \cdot \min \{ 2^\ell \alpha n, (\alpha n)^{k/(k-1)}\}$ with the choice of $c_2=c_k/2$.

Therefore, we have the existence of the desired cover for all $1 \le \ell \le \log K $ from our choices of $\calH_\ell$ in each regime. Choosing $C= \max \{C_1, C_2 \}$ and $c= \min \{ c_1, c_2 \} = \gamma c_0^k /2$ concludes the proof of the theorem.

\end{proof}

\subsection{Proof of the upper bound of Theorem~\ref{thm:ind}}

In this subsection, we establish an upper bound on the $k$-partite hole number of a random Cayley hypergraph of uniformity $k$. Using this bound, we prove the existence of a large non-$k$-fold sumset in a given finite abelian group.

\begin{proof}[Proof of Theorem~\ref{thm:ind}]
    Let $p\in (0, 1/2]$ be a fixed real number and $S$ be a $p$-random subset of $\Gamma$. Let $\ell_0$ be an integer such that 
    $$
        \ell_0 \in [\log_2 (p^{-1/(2k-2)}(\log n )^{3/4})-1, \log_2 (p^{-1/(2k-2)}(\log n )^{3/4})+1].
    $$ 
    This choice of $\ell_0$ is intended to optimize the right-hand side of (\ref{eq:1}). From Lemma~\ref{lem:cov1} and Lemma~\ref{lem:cov2}, we have the existence of a small cover for the collection of all $k$-partite holes of fixed size. Denote the Cayley sum hypergraph of $G$ generated by $S$ as $H_S$. We bound the probability that there exists a $k$-partite hole of size $s$ in $H_S$ as below.

    \begin{align}\label{eq:1}
        \PP[\alpha_k(H_S)\ge s] \leq \sum_{\ell<\ell_0}\abs{\calF_\ell}(1-p)^{\min_{F\in\calF_\ell}\abs{F_\ell}} + \sum_{\ell\geq \ell_0}\abs{\calH_\ell}(1-p)^{\min_{H\in\calH_\ell}\abs{H_\ell}} 
    \end{align}
        
    Choose $s=\xi p ^{-\frac{2k-1}{2k-2}}(\log n )^{10k}$. We begin with bounding the first sum with our choice of $s$. 
\allowdisplaybreaks
    \begin{align*}
        \sum_{\ell<\ell_0}\abs{\calF_\ell}(1-p)^{\min_{F\in\calF_\ell}\abs{F_\ell}} & \le  \sum_{\ell<\ell_0}   \exp(C2^{2\ell}(\log n)^2)\exp \left(-\frac{pc2^\ell s}{(\ell+1)^2} \right) \\
        & \leq \sum_{\ell < \ell_0}\exp(-C2^{2\ell}(\log n)^{3/2}((\log n)^{1/2}-1))\\
        & \leq \exp(\log \log n -C'(\log n)^{2})\\
        & = o \left(\frac{1}{n} \right).
    \end{align*}

    Note that the main contribution of the first sum comes from the case when $\ell = 0$, and therefore the second inequality is satisfied whenever $s>C(\log n)^4  / pc $.

    It remains to bound the second sum, which is shown to be at most 
    \[
    \sum_{\ell\geq \ell_0}\exp(C(\min  \{(2^\ell s)^{1/k}, s^{1/(k-1)}\}(\log n)^{(k+1)/k})\exp \left(-\frac{pc\min \{2^\ell s, s^{k/(k-1)}\}}{(\ell+1)^2} \right)
    \]
    
    by Lemma~\ref{lem:cov2}. Define the set of integers $L_1$ and $L_2$ as below to divide the cases into the case when the minimum is chosen to be $2^\ell s$ and $s^{k/(k-1)}$, respectively.
    \[
    L_1 \defeq \{ \ell \ge \ell_0 : 2^{\ell}s \le s^{k/(k-1)}\} ,
    \ L_2 \defeq \{ \ell \ge \ell_0 : 2^{\ell}s \ge s^{k/(k-1)}\}.
    \]
    Let $C_1'$ and $C_2'$ be constants depending only on $k$. Then the contribution from $\ell \in L_0$ is bounded as
    \allowdisplaybreaks
    \begin{align*}
        \sum_{\ell \in L_1}&\exp(C(2^\ell s)^{1/k} (\log n )^{(k+1)/k})  \exp(-\frac{pc2^\ell s}{(\ell+1)^2})\\ 
        & \leq \sum_{\ell\geq \ell_0}\exp \left(C(2^\ell s)^{1/k}\bigg((\log n)^{(k+1)/k}-\frac{pc}{C}\frac{2^{\ell(k-1)/k}s^{(k-1)/k}}{(\log n)^3}\bigg) \right)\\
        & \leq \exp(\log \log n - C_1'(\log n)^2)\\
        & = o \left(\frac{1}{n} \right).
    \end{align*}

    The contribution from $\ell \in L_1$ is controlled by our choice of $s$. Formally,
    \allowdisplaybreaks
    \begin{align*}
        \sum_{\ell \in L_2}&\exp(C s^{1/(k-1)} (\log n )^{(k+1)/k})  \exp(-\frac{pcs^{k/(k-1)}}{(\ell+1)^2})\\ 
        & \leq \sum_{\ell\in L_2}\exp\left(C(\alpha n)^{1/(k-1)}\bigg((\log n)^{(k+1)/k}-\frac{pc}{C}\frac{s}{(\log n)^3}\bigg)\right)\\
        & \leq \exp(\log \log n - C_2'(\log n)^2)\\
        & = o \left(\frac{1}{n} \right).
    \end{align*}

Since $C, c$ are fixed constants, both of the inequalities hold with a choice of $\xi$ that is sufficiently large. This completes the proof.
    
\end{proof}

We are now ready to prove the upper bound of Theorem~\ref{thm:main}.

\begin{proof}[Proof of the upper bound of Theorem~\ref{thm:main}]
    We show the theorem by proving the existence of a large set that is not a $k$-fold sumset by a probabilistic argument. Let $S_1$ be a $p$-random subset of $\Gamma$ with $p<1/2$ and define $A$ as $\Gamma \backslash (S_1\cup S_2)$. We claim that there exists a choice of small $S_2$ such that $\Gamma \backslash (S_1\cup S_2)$ is not a $k$-fold sumset. Denote the $k$-uniform Cayley sum hypergraph of $G$ generated by $S_1$ as $H_{S_1}$. As in the $2$-uniform case, if $A=kB$ for some $B\subseteq \Gamma$, then $B$ forms an independent set in the $k$-uniform Cayley sum hypergraph of $\Gamma $ with respect to $S_1$. 

    From Theorem~\ref{thm:ind}, the independence number of $H_{S_1}$ is at most $O(p^{-\frac{2k-1}{2k-2}}(\log n)^{10k})$. Using the Chernoff bound as in the proof of Theorem~\ref{thm:ind}, we also observe that $\abs{S_1}$ is at most $2pn$ with high probability. Therefore, we bound the number of possible choices of $B$ by

    $$ \sum_{i\leq \xi p^{-\frac{2k-1}{2k-2}}(\log n)^{10k}}\binom{\abs{G}}{i} \leq \exp{\left(\xi p^{-\frac{2k-1}{2k-2}}(\log n)^{10k+1}\right)} . $$

    Since $p<1/2$, we have at least $\Omega(\exp(pn \log n))$ choices for $S_2$, and as in the proof in \cite{AH25-arxiv}, for sufficiently large $\xi'$, the choice of $p = \xi' n^{-\frac{2k-2}{4k-3}} (\log n)^{5k^2} $ guarantees that there exists a choice of $S_2$ such that $A=\Gamma \backslash(S_1\cup S_2)$ is not a sumset. Therefore, there exists a subset of $\Gamma$ of size at least $n-O(n^{\frac{2k-1}{4k-3}}(\log n)^{5k^2})$ which cannot be represented as a $k$-fold sumset of $\Gamma$. This completes the proof.
\end{proof}

\section{Concluding Remarks}\label{sec:con}

Though we provided both the upper and the lower bound for $f_k(\Gamma)$, the gap between them remains large. Consider $A \subseteq \Gamma$ such that $0 \in A$ holds. In this case, the number of elements $g \in \Gamma$ that have a representation of the form $g = a_1 + \cdots + a_k$ with $a_i \in A$ for all $i \in [k]$ gets greater as $k$ increases. This leads us to conjecture that the exponent should tend to zero as $k$ grows to infinity. Therefore, we conjecture that the asymptotic behavior of the threshold $f_k(\Gamma)$ will match that of the lower bound. In other words, we conjecture the following.

\begin{conjecture}
For all integer $k \ge 2$ and a finite abelian group $\Gamma$ of size $n$ without a nontrivial element of order dividing $k$, it holds that 
$$
    f_k(\Gamma) = \widetilde{\Theta}(n^{1/k}).
$$
\end{conjecture}

As we appropriately generalized the approach of Alon and Pham to the higher uniformity setting and recovered their bound in the $k=2$ case, we expect an improvement on the upper bound to require additional ideas.

\subsection*{Declaration on the use of generative AI}
This project was initiated in September 2025, and the first draft was completed in March 2026, which was written without the assistance of generative AI. The authors subsequently found an error in the proof of Lemma~\ref{lem:cov2}. In correcting this error, ChatGPT-5.5 Pro suggested applying the Paley--Zygmund inequality in the very sparse regime. ChatGPT-5.6 Sol and Fable 5 were also used for proofreading. Apart from these uses, all mathematical ideas were developed independently by the authors. The authors have verified all AI-assisted content and take full responsibility for the paper.

\subsection*{Acknowledgements}
This research was conducted at IBS during the visit of the first author. The first author thanks IBS for their support and hospitality, and is grateful to her advisor, Joonkyung Lee, for making the visit possible through his support.

The authors thank Jozsef Balogh for valuable discussions and for informing the authors of Theorem~\ref {thm:bls}, which was proven in~\cite{Balogh-Liu-Sharifzadeh-arxiv}.

\providecommand{\MR}[1]{}
\providecommand{\MRhref}[2]{%
  \href{http://www.ams.org/mathscinet-getitem?mr=#1}{#2}
}

    \bibliographystyle{amsplain_initials_nobysame}
    \bibliography{bibfile}


\end{document}